\documentclass[11pt]{article}

\usepackage{amsfonts,amsmath,amsthm,amscd,amssymb,latexsym,amsbsy}

    \usepackage[T2A]{fontenc}
    \usepackage[cp1251]{inputenc}
   \usepackage[russian]{babel}

 \usepackage{srcltx}
\begin{document}


\title{Asymptotic expansions of several series \\ and their application}

\author{Viktor P. Zastavnyi\vspace*{.5cm}\\ {\small Version published in Ukrainian Mathematical Bulletin, vol. 6(2009), \No~4}}


\date{}

\theoremstyle{plain}
\newtheorem{theorem}{Theorem}[section]
\newtheorem{lemma}{Lemma}[section]

\newtheorem{proposition}{Proposition}[section]
\newtheorem{corollary}{Corollary}
\newtheorem{definition}{Definition}[section]
\theoremstyle{definition}
\newtheorem{example}{Example}[section]
\newtheorem{remark}{Remark}[section]
\newcommand{\keywords}{\textbf{Key words. }\medskip}
\newcommand{\subjclass}{\textbf{MSC 2000. }\medskip}
\renewcommand{\abstract}{\textbf{Abstract. }\medskip}
\renewcommand{\refname}{\textbf{References}}
\numberwithin{equation}{section}


\renewcommand{\Re}{\mathop{\rm Re}\nolimits}
\renewcommand{\Im}{\mathop{\rm Im}\nolimits}

\newcommand{\res}{\mathop{\rm res}\nolimits}

\newcommand{\N}{\mathbb{N}}
\newcommand{\R}{\mathbb{R}}
\newcommand{\Z}{\mathbb{Z}}
\renewcommand{\C}{\mathbb{C}}

\newcommand{\agam}{{a,\gamma,\alpha,\mu}}

\maketitle

\begin{abstract}
Asymptotic expansions of series
$\sum_{k=0}^{\infty}\varepsilon^k(k+a)^\gamma e^{-(k+a)^\alpha x}$
and
$\sum_{k=0}^{\infty}\frac{\varepsilon^k(k+a)^\gamma}{(x(k+a)^\alpha+1)^\mu}$
in powers of $x$ as $x\to+0$ are found, where $\varepsilon=1$ or
$\varepsilon=-1$. These expansions are applied to obtain precise
inequalities for Mathieu serieses.
\end{abstract}



\section{Introduction and formulation of main results} \label{subsec1}

One of the purposes of this paper is to investigate functional
series of the form
   \begin{equation}\label{f4}
    f(x,a,\gamma,\alpha):=\sum_{k=0}^{\infty}(k+a)^\gamma e^{-(k+a)^\alpha x}
    \;,\;x>0\,,
    \end{equation}
    \begin{equation}\label{f3}
     \widetilde{f}(x,a,\gamma,\alpha):=\sum_{k=0}^{\infty}(-1)^k(k+a)^\gamma
     e^{-(k+a)^\alpha x} \;,\; x>0\,,
    \end{equation}
with parameters $a>0$, $\gamma\in \R$ and $\alpha>0$.
Series~\eqref{f4} and~\eqref{f3} appear in many problems of the
analysis. In particular, for $x=\ln\frac{1}{\rho}$, $a=\frac 12$,
$\alpha=1$, $\gamma=-r-1$ and $r\in\N$ series~\eqref{f4}
and~\eqref{f3} appeared in the paper due to
A.~F.~Timan~\cite{Timan} in 1950. He proved that these series give
an exact value of the remainder when periodic differentiable
functions are approximated by Poisson integrals. Finding a
complete asymptotic representation was the aim of the papers due
to L.~V.~Malei~\cite{Maley}, \'{E}.~L.~Shtark~\cite{Stark},
V.~A.~Baskakov~\cite{Baskakov}, and K.~M.~Zhigallo and
Yu.~I.~Kharkevich~\cite{Harkevich2002_1}. A complete solution to
this problem was obtained in the author's
paper~\cite{Zast2009_mz}, where expansions in series in powers of
$x$ were found in the explicit form for the functions~\eqref{f4}
and~\eqref{f3} with $a>0$, $\alpha=1$, $\gamma=-r-1$ and
$r\in\Z_+$ if $0<x<2\pi$ and $0<x<\pi$ respectively.

Applying the residue theory, Gel'fond~\cite[\S 4.3]{Gelfond} in
1966 found an asymptotic expansion in powers of $x^k$, $k\in\Z_+$
as $x\to+0$ for the function~\eqref{f4} with $a=1$,
$-\frac{\gamma+1}{\alpha}\not\in\Z_+$. Here $\Z_+:=\N\bigcup\{0\}$
stands for the set of all nonnegative integers. In the case $a=1$,
$-\frac{\gamma+1}{\alpha}\in\Z_+$ he pointed out only that the sum
in an asymptotic expansion {\it has to be changed in an
appropriate way}. In~\cite[\S 4.3]{Gelfond}, it was also claimed
that for $a=1$ an asymptotic expansion of the function~\eqref{f3}
in powers of $x^k$, $k\in\Z_+$ as $x\to+0$  can be obtained
similarly. In 2008, using the Euler--Maclaurin formula it was
found by the author~\cite{Zast2008} the asymptotic expansion of
the functions~\eqref{f4} and~\eqref{f3} as $x\to+0$ for any $a>0$
and $\gamma\in\Z_+$, $\alpha\in\N$ (note that, in examples on
pp.~56--57 of~\cite{Zast2008}, the term $(-1)^{\alpha k+\gamma}$
is missing under the sum sign in the right-hand side of asymptotic
expansions).

In theorems~\ref{th4} and~\ref{th5} here the asymptotic expansions
of the functions~\eqref{f4} and~\eqref{f3} as $x\to+0$ are given
for all admissible parameters. Coefficients of these expansions
are expressed by the Hurwitz function ${\zeta}(s,a)$ and the
function $\widetilde{\zeta}(s,a)$ respectively. The last ones are
defined by formulas
\begin{equation}\label{gurv}
\begin{split}
 &{\zeta}(s,a):=\sum_{k=0}^{\infty}\frac{1}{(k+a)^s}\;,\;\Re s>1\;;
\\&
 \widetilde{\zeta}(s,a):=\sum_{k=0}^{\infty}\frac{(-1)^k}{(k+a)^s}\;,\;\Re
 s>0\;,
\end{split}
\end{equation}
for a fixed $a>0$.

Using the Hermite formula one can continue the Hurwitz function
analytically to $\C\setminus\{1\}$. Moreover, the point $s=1$
stands for its first-order pole, and for $a>0$ the following
relations hold (see~\cite{UW,Bateman}):
\begin{equation}\label{polus}
\begin{split}
 &\lim_{s\to 1}\left({\zeta}(s,a)-\frac{1}{s-1}\right)=-\frac{\Gamma'(a)}{\Gamma(a)}\;,\\
 &\lim_{s\to 1}\left({\zeta}(s,a)-2^{1-s}{\zeta}\left(s,\frac{a+1}{2}\right)\right)=
 -\frac{\Gamma'(a)}{\Gamma(a)}+\frac{\Gamma'\left(\frac{a+1}{2}\right)}{\Gamma\left(\frac{a+1}{2}\right)}+\ln
2\;.
 \end{split}
\end{equation}
Here $\Gamma(s)=\int_{0}^{+\infty}e^{-t}\,t^{s-1}\,dt$, $\Re s>0$
is referred to as the Euler gamma-function. The relation
\begin{equation}\label{gurv1}
 \widetilde{\zeta}(s,a)={\zeta}(s,a)-2^{1-s}{\zeta}\left(s,\frac{a+1}{2}\right)\;,\;\Re s>1\;,
\end{equation}
implies that the function $\widetilde{\zeta}(s,a)$ is analytically
continued to $\C$, and \begin{equation}\label{gurv2}
\widetilde{\zeta}(1,a)=
-\frac{\Gamma'(a)}{\Gamma(a)}+\frac{\Gamma'\left(\frac{a+1}{2}\right)}{\Gamma\left(\frac{a+1}{2}\right)}+\ln
2\;,\;a>0\;.
\end{equation}
\begin{theorem}\label{th4}
Suppose that $a>0$, $\gamma\in \R$, and $\alpha>0$. Then the
following asymptotic expansions hold:
\begin{equation}\label{ass1}
f(x,a,\gamma,\alpha)\underset{x\to+0}{\sim}
\frac{1}{\alpha}\;\Gamma\left(\frac{\gamma+1}{\alpha}\right)\,
      x^{-\frac{\gamma+1}{\alpha}}+
      \sum_{k=0}^{\infty}\frac{(-1)^k}{k!}\,\zeta(-\alpha
      k-\gamma,a)\,x^k\;,
      \end{equation}
      \begin{equation}\label{ass2}
\begin{split}
f(x,a,\gamma,\alpha)\underset{x\to+0}{\sim}
&\;\frac{(-1)^rx^r}{\Gamma(r+1)}\left(-\frac{\ln
x}{\alpha}+\frac{\Gamma'(r+1)}{\Gamma(r+1)}\,\frac{1}{\alpha}-\frac{\Gamma'(a)}{\Gamma(a)}\right)\,
 +\\
  &\sum_{k=0,k\ne r}^{\infty}\frac{(-1)^k}{k!}\,\zeta(-\alpha k-\gamma,a)\,x^k,
\end{split}
\end{equation}
if $-\frac{\gamma+1}{\alpha}\not\in\Z_+$ and
$-\frac{\gamma+1}{\alpha}=r\in\Z_+$ respectively. If $0<\alpha<1$,
then~\eqref{ass1} and~\eqref{ass2} turn into equalities whenever
$x>0$. If $\alpha=1$, then~\eqref{ass1} and~\eqref{ass2} are
equalities whenever $x\in(0,2\pi)$.
\end{theorem}

The fact that the relations~\eqref{ass1} are equalities for all
$x>0$ if $a=1$, $0<\alpha<1$, was mentioned without proof
in~\cite{Gelfond}.

\begin{theorem}\label{th5}
Let $a>0$, $\gamma\in \R$, and let $\alpha>0$.  Then the following
asymptotic expansion holds:
\begin{equation}\label{ass3}
\widetilde{f}(x,a,\gamma,\alpha)\underset{x\to+0}{\sim}
\sum_{k=0}^{\infty}\frac{(-1)^k}{k!}\,\widetilde{\zeta}(-\alpha
k-\gamma,a)\,x^k\;.
\end{equation}
If $0<\alpha<1$, then the~\eqref{ass3} turns into an equality for
all $x>0$. If $\alpha=1$, then~\eqref{ass3} is an equality for all
$x\in(0,\pi)$.
\end{theorem}

In what follows we will consider the functional series of the form
\begin{equation}\label{g2}
   g(x,\agam):=\sum_{k=0}^{\infty}\frac{(k+a)^\gamma}{(x(k+a)^\alpha+1)^\mu}
  \;,\;  \mu>\max\left\{\frac{\gamma+1}{\alpha};0\right\}\,,\,x> 0\,,
\end{equation}
\begin{equation}\label{g1}
  \widetilde{g}(x,\agam):=\sum_{k=0}^{\infty}\frac{(-1)^k(k+a)^\gamma}{(x(k+a)^\alpha+1)^\mu}
   \;,\; \mu>\max\left\{\frac{\gamma}{\alpha};0\right\}\,,\,x> 0\,,
\end{equation}
with parameters $a>0$, $\gamma\in \R$, $\alpha>0$ and $\mu>0$.

\begin{theorem}\label{thg1}
Suppose that $a>0$, $\gamma\in \R$, $\alpha>0$ and
$\mu>\max\left\{\frac{\gamma+1}{\alpha};0\right\}$. Then the
following asymptotic expansions hold:
\begin{equation}\label{assg1}
\begin{split}
&g(x,\agam)\underset{x\to+0}{\sim}
  \\
&\frac{\Gamma\left(\frac{\gamma+1}{\alpha}\right)\;\Gamma
\left(\mu-\frac{\gamma+1}{\alpha}\right)}{\alpha\,\Gamma(\mu)}\,
 x^{-\frac{\gamma+1}{\alpha}}+
      \sum_{k=0}^{\infty}\frac{(-1)^k}{k!}\,\frac{\Gamma(\mu+k)}{\Gamma(\mu)}\,\zeta(-\alpha
      k-\gamma,a)\,x^k\;,
      \end{split}
      \end{equation}
      \begin{equation}\label{assg2}
\begin{split}
&g(x,\agam)\underset{x\to+0}{\sim}
  \\
&\frac{\Gamma(\mu+r)(-1)^rx^r}{\Gamma(\mu)\Gamma(r+1)}\left(-\frac{\ln
x}{\alpha}+\frac{\Gamma'(r+1)}{\alpha\,\Gamma(r+1)}-\frac{\Gamma'(a)}{\Gamma(a)}-\frac{\Gamma'(\mu+r)}{\alpha\,\Gamma(\mu+r)}\right)
 +\\
      &\sum_{k=0,k\ne r}^{\infty}\frac{(-1)^k}{k!}\,\frac{\Gamma(\mu+k)}{\Gamma(\mu)}\,\zeta(-\alpha k-\gamma,a)\,x^k,
\end{split}
\end{equation}
for $-\frac{\gamma+1}{\alpha}\not\in\Z_+$  and
$-\frac{\gamma+1}{\alpha}=r\in\Z_+$ respectively.
\end{theorem}

\begin{theorem}\label{thg2}
Let $a>0$, $\gamma\in \R$, $\alpha>0$, and let
$\mu>\max\left\{\frac{\gamma}{\alpha};0\right\}$. Then the
following asymptotic expansion holds:
\begin{equation}\label{assg3}
\widetilde{g}(x,\agam)\underset{x\to+0}{\sim}
\sum_{k=0}^{\infty}\frac{(-1)^k}{k!}\,\frac{\Gamma(\mu+k)}{\Gamma(\mu)}\,\widetilde{\zeta}(-\alpha
k-\gamma,a)\,x^k\;.
\end{equation}
\end{theorem}
Note that Theorems~\ref{th4}, \ref{th5}, \ref{thg1} and~\ref{thg2}
are proved by using the technique of~\cite{Gelfond}.
In~\S~\ref{subsec4} these theorems are applied to obtain precise
inequalities for Mathieu series.

\section{Precise inequalities for Mathieu series}
\label{subsec4}

Consider the following functional series with parameters $a>0$,
$\gamma\in \R$, $\alpha>0$ and $\mu>0$:
 \begin{equation}\label{f2}
 S(x,\agam):=\sum_{k=0}^{\infty}\frac{(k+a)^\gamma}{((k+a)^\alpha+x)^\mu}
 \;,\;  \mu>\max\left\{\frac{\gamma+1}{\alpha};0\right\}\,,\,x\ge 0\,,
 \end{equation}
       \begin{equation}\label{f1}
      \widetilde{S}(x,\agam):=\sum_{k=0}^{\infty}\frac{(-1)^k(k+a)^\gamma}{((k+a)^\alpha+x)^\mu}
      \;,\;
      \mu>\max\left\{\frac{\gamma}{\alpha};0\right\}\,,\,x\ge 0\,.
      \end{equation}

As is customary in recent years,~\eqref{f2} and~\eqref{f1} are
said to be a generalized Mathieu series and a generalized
alternating Mathieu series respectively. In 1890, \'{E}mile
Leonard Mathieu~\cite{Mathieu} introduced the hypothesis on the
validity of the following inequality:
\begin{equation}\label{Mathieu}
S(x,1,1,2,2)=\sum_{k=0}^{\infty}\frac{k+1}{((k+1)^2+x)^2}<\frac{1}{2x}
\;, x>0\;.
\end{equation}
Different proofs of the inequality~\eqref{Mathieu} were published
in the papers due to Berg~\cite{Berg}, van der Corput,
Heflinger~\cite{Corput} and Makai~\cite{Makai} in 1952--1957. In
the Makai's paper~\cite{Makai} there were proved the inequalities
\begin{equation}\label{Mak}
\frac{1}{2(q+x)}<\sum_{k=0}^{\infty}\frac{k+1}{((k+1)^2+x)^2}<\frac{1}{2(p+x)}
\;, x>0\;,
\end{equation}
where $q=\frac{1}{2}$ and $p=0$. Thus, the following natural
problem arises: to find a maximal possible $p$ and a minimal
possible $q$ satisfying the inequality~\eqref{Mak}. In 1982,
Elbert~\cite{Elbert} conjectured that one can take
$q=\frac{1}{2\zeta(3)}$ in~\eqref{Mak}, where $\zeta(s)$ stands
for the Riemann zeta-function. In 1998, Alzer, Brenner and
Ruehr~\cite{Alzer} proved that $q=\frac{1}{2\zeta(3)}$ and
$p=\frac{1}{6}$ are sharp constants in the inequality~\eqref{Mak}.

In 2008, it was proved by the author~\cite{Zast2009_1} that for
any $\mu>1$ and $a\ge 1$ there exist positive constants $m(\mu,a)$
and $M(\mu,a)$ such that the inequality
     \begin{equation}\label{ner2}
  \frac{1}{2(\mu-1)(q+x)^{\mu-1}}
  \le \sum_{k=0}^{\infty}\frac{(k+a)}{((k+a)^2+x)^{\mu}}\le
  \frac{1}{2(\mu-1)(p+x)^{\mu-1}}
    \end{equation}
is fulfilled for every $x>0$ if and only if $0\le p\le m(\mu,a)$
and $q\ge M(\mu,a)$. In this case, for any fixed $a\ge 1$ the
functions $m(\mu,a)$ and $M(\mu,a)$ decrease and increase
respectively on $\mu\in(1,+\infty)$, and for all $a\ge 1$, $\mu>1$
the following inequalities hold:
\begin{equation*}
\begin{split}
 &a^2-a<m(\infty,a)\le m(\mu,a)\le a^2-a+\frac 16
 \\& a^2-a+\frac 14<M(\mu,a)<M(\infty,a)=a^2\;.
\end{split}
\end{equation*}
It was also proved that $m(\mu,1)=\frac 16$, $\mu\in(1,3]$. Thus,
{\it if $a\ge 1$ then inequality~\eqref{ner2} is valid for all
$\mu>1$ if and only if $0\le p\le m(\infty,a)$ and $q\ge
M(\infty,a)=a^2$.} The right-hand side inequality in~\eqref{ner2}
was proved for $a=1$, $p=0$ and $\mu>1$ by
Diananda~\cite{Diananda} in 1980. A big list related to this
matter can be found in~\cite{Hoorfar and Qi}.

If $a>0$, $\gamma\in \R$, $\alpha>0$ and
$\mu>\max\left\{\frac{\gamma+1}{\alpha};0\right\}$, then by
Theorem~\ref{thg1} we have
 \begin{equation}\label{ass}
  S(x,\agam)\underset{x\to+\infty}{\sim}\left\{
  \begin{array}{lcr}
   \frac{\Gamma\left(\frac{\gamma+1}{\alpha}\right)}{\alpha}\cdot
    \frac{\Gamma\left(\mu-\frac{\gamma+1}{\alpha}\right)}{\Gamma(\mu)}\cdot
   x^{\frac{\gamma+1}{\alpha}-\mu}
    &,&\gamma+1>0\;,\\
    \frac{1}{\alpha}\cdot x^{-\mu}\ln x
    &,&\gamma+1=0\;,\\
    \zeta(-\gamma,a)\,x^{-\mu}
    &,&\gamma+1<0\;.
      \end{array}
      \right.
      \end{equation}
Hence the following problem is natural. Assume that $a>0$,
$\gamma+1>0$, $\alpha>0$ and $\mu_0\ge\frac{\gamma+1}{\alpha}$.
For which $q\ge0$, $p\ge0$, $A\in\R$ and $B>0$ the inequality
 \begin{equation}\label{n2}
 \frac{B\cdot\Gamma\left(\mu-\frac{\gamma+1}{\alpha}
   \right)}{\Gamma(\mu)(q+x)^{\mu-\frac{\gamma+1}{\alpha}}}\le
   S(x,\agam)\le
  \frac{A\cdot\Gamma\left(\mu-\frac{\gamma+1}{\alpha}
  \right)}{\Gamma(\mu)(p+x)^{\mu-\frac{\gamma+1}{\alpha}}}
 \end{equation}
holds for any $\mu>\mu_0$ and $x>0$? This problem is completely
solved in Theorem~\ref{th1}, and an analogous problem is solved in
Theorem~\ref{th2} in the case of $\gamma+1<0$.
      \begin{theorem}\label{th1}
Suppose that $a>0$, $\gamma+1>0$, $\alpha>0$,
$\mu_0\ge\frac{\gamma+1}{\alpha}$, $q\ge0$, $p\ge0$, $A\in\R$ and
$B>0$. Then the inequality~\eqref{n2} is valid for all $\mu>\mu_0$
and $x>0$ if and only if $0\le p< a^\alpha\le q$, $A\ge
A_p(a,\gamma,\alpha)$, $0<B\le B_q(a,\gamma,\alpha)$, where
      \begin{equation}
      \begin{split}
&A_p(a,\gamma,\alpha):=\sup_{x>0}e^{p\,x}x^{\frac{\gamma+1}{\alpha}}f(x,a,\gamma,\alpha)\;,
 \\&
B_q(a,\gamma,\alpha):=\inf_{x>0}e^{q\,x}x^{\frac{\gamma+1}{\alpha}}f(x,a,\gamma,\alpha)\,.
      \end{split}
      \end{equation}
In this case, the inequality~\eqref{n2} is strict for all $x>0$.
If $p>0$, it is also strict at $x=0$. Moreover,
$A_p(a,\gamma,\alpha)<+\infty$ if and only if $p<a^\alpha$, and
$B_q(a,\gamma,\alpha)>0$ if and only if $q\ge a^\alpha$. If $a\ge
1$ then $A_p(a,1,2)=B_q(a,1,2)=\frac 12$ for all $q\ge a^2$ and
$0\le p\le m(\infty,a)$; in particular, this is valid for all
$p\in[0,a^2-a]$.
      \end{theorem}

       \begin{theorem}\label{th2}
Assume that $a>0$, $\gamma+1<0$, $\alpha>0$,  $\mu_0\ge0$,
$q\ge0$, $p\ge0$, $D\in\R$ and $E>0$. Then the inequality
      \begin{equation}\label{n3}
      \frac{E}{(q+x)^{\mu}}\le
      S(x,\agam)\le
      \frac{D}{(p+x)^{\mu}}
      \end{equation}
holds for every $\mu>\mu_0$ and $x>0$ if and only if $0\le p\le
a^\alpha\le q$, $D\ge D_p(a,\gamma,\alpha)$, $0<E\le
E_q(a,\gamma,\alpha)$, where
      \begin{equation}
      \begin{split}
&D_p(a,\gamma,\alpha):=\sup_{x>0}e^{p\,x}f(x,a,\gamma,\alpha)\;,
\\&
E_q(a,\gamma,\alpha):=\inf_{x>0}e^{q\,x}f(x,a,\gamma,\alpha)\,.
\end{split}
\end{equation}
In this case, inequality~\eqref{n3} is strict for all $x>0$. If
$p>0$, it is also strict at $x=0$. Moreover,
$D_p(a,\gamma,\alpha)<+\infty$ if and only if $p\le a^\alpha$, and
$E_q(a,\gamma,\alpha)>0$ if and only if $q\ge a^\alpha$. Besides,
$D_p(a,\gamma,\alpha)=\zeta(-\gamma,a)$ for all $p\le a^\alpha$
and $E_q(a,\gamma,\alpha)=a^\gamma$ at $q= a^\alpha$.
      \end{theorem}

If $a>0$, $\gamma\in \R$, $\alpha>0$ and
$\mu>\max\left\{\frac{\gamma}{\alpha};0\right\}$, then
Theorem~\ref{thg2} yields that
      \begin{equation}
      \widetilde{S}(x,\agam)=x^{-\mu}\left(\widetilde{\zeta}(-\gamma,a)+o(1)\right)\;,\;
      x\to+\infty\;.
      \end{equation}
Therefore the following problem is natural for series~\eqref{f1}.
Let $a>0$, $\gamma\in \R$, $\alpha>0$, and let
$\mu_0\ge\max\{\frac{\gamma}{\alpha};0\}$. For which $q\ge 0$,
$p\ge 0$ and $C,F\in \R$ the inequality
      \begin{equation}\label{n1}
      \frac{F}{(q+x)^\mu}\le\widetilde{S}(x,\agam)\le\frac{C}{(p+x)^\mu}
      \end{equation}
holds for any $\mu>\mu_0$ and $x>0$? In the paper due to Tomovski
and Hilfer~\cite{Tomovski_Hilfer}, it is claimed that this is
satisfied in the case $a=1$, $\gamma>0$ if we take $p=C=1$ and
$\mu_0=\frac{\gamma+1}{\alpha}$ in the right-hand side of
the~\eqref{n1}. A mistake in the proof of this Tomovski and
Hilfer's assertion was indicated by the author in~\cite{Zast2009}.
In the same paper~\cite{Zast2009}, it was proved that for
$m,\alpha\in N$, $\gamma=4m+5$, $\alpha\mu-\gamma>0$ the
right-hand side inequality of the~\eqref{n1} with $a=p=C=1$ is
impossible for large $x>0$. Theorem~\ref{th3} presents all the
solutions to this problem.
      \begin{theorem}\label{th3}
Suppose that $a>0$, $\gamma\in \R$, $\alpha>0$,
$\mu_0\ge\max\{\frac{\gamma}{\alpha};0\}$, $q\ge 0$, $p\ge 0$ and
$C,F\in \R$. Then the inequality~\eqref{n1} is satisfied for every
$\mu>\mu_0$ and $x>0$ if and only if $0\le p\le a^\alpha$,  $C\ge
C_p(a,\gamma,\alpha)$,  $F\le F_q(a,\gamma,\alpha)$, where
      \begin{equation}
      \begin{split}
      &C_p(a,\gamma,\alpha):=\sup_{x>0}e^{p\,x}\widetilde{f}(x,a,\gamma,\alpha)\;,
      \\&
      F_q(a,\gamma,\alpha):=\inf_{x>0}e^{q\,x}\widetilde{f}(x,a,\gamma,\alpha).
      \end{split}
      \end{equation}
In this case, the inequality~\eqref{n1} is strict for $x>0$. If
$q,p>0$, it is also strict at $x=0$. Moreover, we have
$0<C_p(a,\gamma,\alpha)<+\infty$ for $p\le a^\alpha$ and
$C_p(a,\gamma,\alpha)=+\infty$ for $p> a^\alpha$.
      \end{theorem}

\section{Preliminaries}
\label{subsec2}
\subsection{Euler gamma-function}
The function $\Gamma(s)$ is analytically continued to the whole
plane $\C$ except to the points $s=-k$, $k\in\Z_+$ in which it has
simple poles. Moreover, the following relations
\begin{equation}\label{11eyler}
 \Gamma(s+1)=s\Gamma(s)\;\;,\;\;\Gamma(s)\Gamma(1-s)=\frac{\pi}{\sin\pi s}\;
\end{equation}
hold for all admissible $s\in\C$. If $s=\sigma+it=|s|e^{i\varphi}$, where $\sigma,t\in\R$ and
$\varphi=\varphi(s)=\arg s\in(-\pi,\pi)$, we have
\begin{equation}\label{11eyler_a}
\left.
\begin{split}
&\Gamma(s+1)=\sqrt{2\pi}s^{s+\frac
12}e^{-s}e^{R(s)}\;\;\;,\;\;\;|R(s)|\le\frac{1}{12|s|\cos^2\frac{\varphi}{2}}\;,
  \\
&|s^{s+\frac 12}|=|s|^{\sigma+\frac
12}e^{-\varphi(s)t}=|s|^{\sigma+\frac 12}e^{-|\varphi(s)|\,
|t|}\;,
  \\
&|\Gamma(s+1)|=\sqrt{2\pi}|s|^{\sigma+\frac
12}e^{-\sigma}e^{-|\varphi(s)|\,  |t|}\,|e^{R(s)}|\;
\end{split}
\right\}
\end{equation}
(see~\cite[\S 1.5.1]{Riekst}). If, in addition, $\Re s=\sigma>0$, then
$|\varphi(s)|=\arctg\frac{|t|}{\sigma}$, and hence
\begin{equation}\label{12eyler}
\left.
\begin{split}
&|\Gamma(s+1)|=\sqrt{2\pi}|s|^{\sigma+\frac
12}e^{-\sigma}e^{-|t|\arctg\frac{|t|}{\sigma}}\,|e^{R(s)}|\;,\;|R(s)|\le\frac{1}{6|s|}\;,
  \\
&|\Gamma(s+1)|\le\sqrt{2\pi}|s|^{\sigma+\frac
12}e^{-\frac{\pi}{2}|t|}\,e^{\frac{1}{6|s|}}\;,
  \\
&\frac{1}{|\Gamma(s+1)|}\le(2\pi)^{-\frac 12}|s|^{-\sigma-\frac
12}e^{\sigma}e^{|t|\arctg\frac{|t|}{\sigma}}\,e^{\frac{1}{6|s|}}\;.
\end{split}
\right\}
\end{equation}
Here we take  the inequalities $|e^{ w}|\le e^{|w|}$, $w\in\C$ and
$0<\frac{\pi}{2}-\arctg u<\frac 1u$, $u>0$ into account.

Let $0<\delta\le\frac{\pi}{2}$, $|\arg s|\le\pi-\delta$, and let
$\Re s=\sigma$. Considering the cases $\sigma>0$ and $\sigma\le 0$
(here $|\arg s|\ge\frac{\pi}{2}$) separately we derive
from~\eqref{11eyler_a} that
\begin{equation}\label{13eyler}
|\Gamma(s+1)|\le\sqrt{2\pi}|s|^{\sigma+\frac
12}\,e^{\frac{|\sigma|-\sigma}{2}}\,e^{-\frac{\pi}{2}|t|}\,
e^{\frac{1}{12|s|\sin^2\frac{\delta}{2}}}\;,\; |\arg
s|\le\pi-\delta\;.
\end{equation}

\subsection{Hurwitz function}
If the case of $a=p+a_0$, where $p\in\N$, $0<a_0\le 1$, we have
\begin{equation}\label{09eyler}
  \zeta(s,a)=\zeta(s,a_0)-\sum_{k=0}^{p-1}\frac{1}{(k+a_0)^s}\;,\;s\ne 1\;.
\end{equation}
The relation~\eqref{09eyler} is obvious whenever $\Re s>1$, and
for the remaining $s\ne 1$ it is implied by the uniqueness theorem
for analytic functions. The following formula is due to Hurwitz:
\begin{equation}\label{10eyler}
\zeta(s,a)=\frac{2\Gamma(1-s)}{(2\pi)^{1-s}}\;\sum_{k=1}^{\infty}\frac{\sin\left(2\pi
ak+\frac{\pi }{2}s\right)}{k^{1-s}}\;,\;\Re s <0\;,\;0<a\le 1\;.
\end{equation}
If $0<a\le 1$ then it follows from~\cite[\S 13.51]{UW} the
existence of positive constants $c(a)>0$ and $t(a)>1$ such that
the inequality
\begin{equation}\label{14eyler}
\left.
\begin{split}
  &|\zeta(\sigma+it,a)|\le c(a)|t|^{\tau(\sigma)}\ln|t|\;,\;|t|\ge t(a)\;,\;\text{where}
  \\
  &\tau(\sigma):=\left\{
 \begin{array}{ccl}
 \frac 12-\sigma&,&\sigma\le 0\;,\\
 \frac 12&,&0\le \sigma\le \frac 12\;,\\
  1-\sigma&,&\frac 12\le \sigma\le 1\;,\\
   0&,& \sigma\ge 1\;
  \end{array}
  \right.
\end{split}
\right\}
\end{equation}
is fulfilled for $\sigma, t\in\R$.

\subsection{Mellin transform}
If $x^{\sigma-1}f(x)\in L(0,+\infty)$ for some $\sigma\in\R$, then
the Mellin transform of a function $f$ is defined by
\begin{equation*}
g(s)=\int_{0}^{+\infty}x^{s-1}f(x)\,dx\;,\;s=\sigma+it\;,\;t\in\R\;.
\end{equation*}
If, in addition, the function $f$ is of bounded variation in a
neighborhood of $x>0$, then the following inversion formula holds:
\begin{equation*}
  \frac{f(x+0)+f(x-0)}{2}= \frac{1}{2\pi i}
  \int_{\sigma-i\infty}^{\sigma+i\infty}g(s)x^{-s}\,ds\;
\end{equation*}
(see~\cite[\S 1.29]{Titch}), where the integral is treated in the
sense of Cauchy principal value. If we take $f(x)=e^{-x}$ and
$f(x)=(x+1)^{-\mu}$, $\mu>0$, we find $g(s)=\Gamma(s)$, $\Re s>0$,
and $g(s)=\frac{\Gamma(\mu-s)\Gamma(s)}{\Gamma(\mu)}$, $0<\Re
s<\mu$, respectively. Therefore for any  $x>0$ the following two
relations are met:
\begin{equation}\label{melin1}
e^{-x}=\frac{1}{2\pi i}
\int_{\sigma-i\infty}^{\sigma+i\infty}\Gamma(s)x^{-s}\,ds\;,\;\sigma>0\;,
\end{equation}
\begin{equation}\label{melin2}
\frac{1}{(x+1)^{\mu}}=\frac{1}{2\pi i}
\int_{\sigma-i\infty}^{\sigma+i\infty}
\frac{\Gamma(\mu-s)\Gamma(s)}{\Gamma(\mu)}\,x^{-s}\,ds\;,\;0<\sigma<\mu\;.
\end{equation}

\section{Proofs of Theorems~\ref{th4}, \ref{th5}, \ref{thg1} and~\ref{thg2} }
\label{subsec3}
\subsection*{Proof of asymptotic expansions in Theorem~\ref{th4}}
Take $\sigma=\beta>\max\left\{0,\frac{\gamma+1}{\alpha}\right\}$
in~\eqref{melin1}, and replace $x$ by $(k+a)^{\alpha}x$, $a>0$,
$k\in\Z_+$, $\alpha>0$, $x>0$. Then we summarize the obtained
inequalities
\begin{equation*}
(k+a)^{\gamma}e^{-(k+a)^{\alpha}x}=\frac{1}{2\pi i}
\int_{\beta-i\infty}^{\beta+i\infty}\Gamma(s)x^{-s}(k+a)^{-\alpha
s+\gamma}\,ds
\end{equation*}
over all $k\in\Z_+$. In the left-hand side, we obtain
$f(x,a,\gamma,\alpha)$. In the right-hand side, we interchange the
sum and the integral signs (this is well defined in view
of~\eqref{11eyler_a},~\eqref{12eyler} and $\alpha\beta-\gamma>1$).
We obtain that for any $a>0$, $\gamma\in \R$,  $\alpha>0$, $x>0$
and $\beta>\max\left\{0,\frac{\gamma+1}{\alpha}\right\}$ the
following relation holds:
\begin{equation}\label{15eyler}
f(x,a,\gamma,\alpha)= \frac{1}{2\pi i}
\int_{\beta-i\infty}^{\beta+i\infty}F(s)\,ds\;,\;
 F(s)=\Gamma(s)x^{-s}\zeta(\alpha s-\gamma,a)\;.
\end{equation}
The function $F(s)$ is analytic on the whole plane except in the
poles $s=-k$, $k\in\Z_+$ and $s=\frac{\gamma+1}{\alpha}$. If
$-\frac{\gamma+1}{\alpha}\not\in\Z_+$ then these poles are
different and simple. Take $\sigma_n=n+\frac 12$, $n\in\Z_+$, and
consider the rectangle
$$
K_{n,m}:=\left\{s\in\C:|\Im s|\le m\,,\,-\sigma_n\le\Re
s\le\beta\right\}\;,\;m\in\N\;,\;n\in\Z_+\;.
$$
If $\sigma_n\ne-\frac{\gamma+1}{\alpha}$ then by the residue
theorem we have
\begin{equation}\label{16eyler}
\frac{1}{2\pi i}\oint_{\partial K_{n,m}}F(s)\,ds=\Sigma_n\;,
\end{equation}
where $\Sigma_n$ stands for the sum of residues of the function
$F$ in poles lying on the interval $(-\sigma_n,\beta)$. If
$n>-\frac 12-\frac{\gamma+1}{\alpha}$, $n\in\Z_+$, then the
interval $(-\sigma_n,\beta)$ contains only poles $s=-k$,
$k=0,\ldots,n$, and $s=\frac{\gamma+1}{\alpha}$. Both
estimates~\eqref{13eyler}, \eqref{14eyler} and
relation~\eqref{09eyler} yield (in the case of $a>1$) that the
integrals over horizontal segments $s=\sigma\pm im$,
$-\sigma_n\le\sigma\le\beta$  tend to zero as $m\to+\infty$ in the
left-hand side of~\eqref{16eyler}. Therefore both~\eqref{15eyler}
and~\eqref{16eyler} imply the following relation:
\begin{equation}\label{17eyler}
f(x,a,\gamma,\alpha)=\Sigma_n+I_n\;,\;
I_n=\frac{1}{2\pi i}
\int_{-\sigma_n-i\infty}^{-\sigma_n+i\infty}F(s)\,
ds\;,\;\sigma_n\ne-\frac{\gamma+1}{\alpha}\;.
\end{equation}
First, let us calculate $\Sigma_n$. It follows from
relation~\eqref{polus} that the expansion of the function
$\zeta(\alpha s-\gamma,a)$ in a Laurent series in a neighborhood
of the pole $s=\frac{\gamma+1}{\alpha}$ can be written as
\begin{equation}\label{18eyler}
\begin{split}
&\zeta(\alpha
s-\gamma,a)=\frac{c_{-1}}{s-\frac{\gamma+1}{\alpha}}+
c_0+c_1\left(s-\frac{\gamma+1}{\alpha}\right)+\ldots\;,\;
 \\&
c_{-1}=\frac{1}{\alpha}\;,\; c_0=-\frac{\Gamma'(a)}{\Gamma(a)}\;.
\end{split}
\end{equation}
Complement formula~\eqref{11eyler} implies that the expansion of
the $\Gamma$-function in a Laurent series  in a neighborhood of
the pole $s=-k$, $k\in\Z_+$ is of the form
\begin{equation}\label{19eyler}
\begin{split}
&\Gamma(s)=\frac{a_{-1}}{s+k}+a_0+a_1\left(s+k\right)+\ldots\;,\;
\\& a_{-1}=\frac{(-1)^k}{\Gamma(k+1)}\;,\;
a_0=\frac{(-1)^k\Gamma'(k+1)}{\Gamma^2(k+1)}\;.
\end{split}
\end{equation}
Thus if $-\frac{\gamma+1}{\alpha}\not\in\Z_+$ then for $n>-\frac
12-\frac{\gamma+1}{\alpha}$, $n\in\Z_+$, the following equality
holds:
\begin{equation}\label{20eyler}
\Sigma_n= \frac{1}{\alpha}\;
\Gamma\left(\frac{\gamma+1}{\alpha}\right)\,
x^{-\frac{\gamma+1}{\alpha}}+
\sum_{k=0}^{n}\frac{(-1)^k}{k!}\,\zeta(-\alpha k-\gamma,a)\,x^k\;.
\end{equation}

If $-\frac{\gamma+1}{\alpha}=r\in\Z_+$ then for
$n>-\frac 12-\frac{\gamma+1}{\alpha}$, $n\in\Z_+$, the relation
\begin{equation}\label{21eyler}
\Sigma_n=
  \underset{s=-r}\res F(s)+ \sum_{k=0,k\ne r}^{n}\underset{s=-k}\res F(s)\;,
\end{equation}
is met, where the residues at the points $s=-k$, $k\in\Z_+$, $k\ne
r$, are calculated as above. To calculate the residue of the
function $F$ at the point $s=-r$ one should take account of the
following expansion of the function $x^{-s}$ in a Taylor series in
a neighborhood of the point $s=-r$:
\begin{equation*}
x^{-s}=b_0+b_1(s+r)+\ldots\;,\;b_0=x^r\;,\;b_1=-x^r\ln x\;.
\end{equation*}
In view of relations~\eqref{18eyler} and~\eqref{19eyler} for
$\frac{\gamma+1}{\alpha}=-r$ and $k=r$ respectively, we obtain the
following expansion of the function $F$ in a Laurent series in a
neighborhood of the pole $s=-r$:
\begin{equation}\label{21eyler_a}
\left.
\begin{split}
F(s)=&\frac{B_{-2}}{(s+r)^2}+\frac{B_{-1}}{(s+r)}+B_0+B_1(s+r)+\ldots\,,
\\
B_{-2}=&c_{-1}a_{-1}b_0=\frac{1}{\alpha}
\cdot\frac{(-1)^r}{\Gamma(r+1)}\,\,  x^r\;,
  \\
B_{-1}=&c_{-1}a_{-1}b_{1}+c_{-1}a_{0}b_{0}+c_{0}a_{-1}b_{0}\;.
\end{split}
  \right\}
\end{equation}
Therefore,
 \begin{equation}\label{22eyler}
 \underset{s=-r}\res F(s)=B_{-1}=
\frac{(-1)^rx^r}{\Gamma(r+1)}\left(-\frac{\ln
x}{\alpha}+\frac{\Gamma'(r+1)}{\Gamma(r+1)}\,\frac{1}{\alpha}-\frac{\Gamma'(a)}{\Gamma(a)}\right)\,.
\end{equation}

Now let us find estimate for the integral $I_n$ in~\eqref{17eyler}. If
$\sigma_n\ne-\frac{\gamma+1}{\alpha}$ we have:
\begin{equation}\label{23eyler}
|I_n|\le\frac{x^{\sigma_n}}{2\pi} \int_{-\infty}^{+\infty}
|\Gamma(-\sigma_n+it)|\,|\zeta(-\alpha\sigma_n-\gamma+i\alpha
t,a)|\,dt\;,\;n\in\Z_+\;,\;x>0\;.
\end{equation}
The convergence of the integral in~\eqref{23eyler} is implied by
both the relation
$$
|\Gamma(-\sigma_n+it)|=\frac{\pi}{\ch(\pi t)\,|\Gamma(1+\sigma_n-it)|}
$$
and estimates~\eqref{12eyler} and~\eqref{14eyler} (in the case of
$a>1$ it should also be considered the relation~\eqref{09eyler}).
Thus asymptotic expansions~\eqref{ass1} and~\eqref{ass2} are
proved.

\subsection*{Case of $0<\alpha\le 1$ in Theorem~\ref{th4}}
For fixed $0<\alpha\le 1$ and $\gamma\in\R$, and an arbitrary
$\varepsilon>0$ we put:
$$
n(\varepsilon,\gamma,\alpha):=\max\left\{\frac{1-\gamma}{\alpha}-\frac
12;\frac{|\gamma|}{\varepsilon}-\frac 12;1\right\}\,.
$$
Then, for all positive integers $n\ge
n(\varepsilon,\gamma,\alpha)$, the following inequalities hold:
\begin{equation*}
\alpha\sigma_n+\gamma\ge 1\;;\;|\gamma|\le \varepsilon\sigma_n\le
\varepsilon|\sigma_n-it|\,,\,t\in\R\;;\;\sigma_n>1\;.
\end{equation*}
If $0<a\le 1$ then we conclude from~\eqref{10eyler}
and~\eqref{12eyler} that, for all positive integers $n\ge
n(\varepsilon,\gamma,\alpha)$ and $t\in\R$, there hold
inequalities (it should also be taken into account that $|\sin
w|\le e^{|\Im w|}$, $w\in\C$):
\begin{equation*}
 \begin{split}
  &|\zeta(-\alpha\sigma_n-\gamma+i\alpha t,a)|
  \le
  \\&
   \frac{2|\Gamma(1+\alpha\sigma_n+\gamma-i\alpha
   t)|}{(2\pi)^{1+\alpha\sigma_n+\gamma}}\,
  e^{\frac{\pi}{2}\alpha|t|}\zeta(1+\alpha\sigma_n+\gamma)\le
  \\
  &\frac{C(\gamma)}{(2\pi)^{\alpha\sigma_n}}\,
  e^{\frac{\pi}{2}\alpha|t|}
  \,|\alpha\sigma_n+\gamma-i\alpha  t|^{\alpha\sigma_n+\gamma+\frac 12}
  e^{-\alpha\sigma_n}
  e^{-\alpha|t|\arctg\frac{\alpha|t|}{\alpha\sigma_n+\gamma}}
  \le
  \\
  &\frac{C(\gamma)}{(2\pi)^{\alpha\sigma_n}}\,
  e^{\frac{\pi}{2}\alpha|t|-\alpha|t|\arctg
  \frac{\alpha|t|}{\alpha\sigma_n+\gamma}}
  e^{-\alpha\sigma_n}
  (\alpha+\varepsilon)^{\alpha\sigma_n+\gamma+\frac 12}
  |\sigma_n-it|^{\alpha\sigma_n+\gamma+\frac 12}\;,
 \end{split}
\end{equation*}
where $C(\gamma)=\frac{2\zeta(2)
e^{\frac{1}{6}-\gamma}}{(2\pi)^{\gamma+1/2}}$, and
$\zeta(s):=\zeta(s,1)$ stands for the Riemann zeta-function. Here
we use the inequalities
$$
\zeta(1+\alpha\sigma_n+\gamma)\le\zeta(2)\,,
$$
$$
|\alpha\sigma_n+\gamma-i\alpha  t|\le \alpha
|\sigma_n-it|+|\gamma|\le (\alpha+\varepsilon)|\sigma_n-it|\,.
$$
Since for all $n\ge n(\varepsilon,\gamma,\alpha)$ and $t\in\R$ the
inequality  $|\sigma_n-it|\ge\sigma_n\ge 1$ is satisfied, we have
for those $n$ and $t$ that
$$
|\sigma_n-it|^{\gamma}\le |\sigma_n-it|^{|\gamma|}\le
\sigma_n^{|\gamma|}\left(1+\frac{|t|}{\sigma_n}\right)^{|\gamma|}\le
\sigma_n^{|\gamma|}\left(1+|t|\right)^{|\gamma|}\;.
$$
Finally, we obtain that for any $n\ge
n(\varepsilon,\gamma,\alpha)$ and $t\in\R$ there holds the
inequality:
\begin{equation}\label{26eyler}
 \begin{split}
  &|\zeta(-\alpha\sigma_n-\gamma+i\alpha t,a)|\le
  \frac{C(\gamma)}{(2\pi)^{\alpha\sigma_n}}\,
  e^{\frac{\pi}{2}\alpha|t|-\alpha|t|\arctg
  \frac{\alpha|t|}{\alpha\sigma_n+\gamma}}
  e^{-\alpha\sigma_n}
  \times\\&
  (\alpha+\varepsilon)^{\alpha\sigma_n+\gamma+\frac 12}
  |\sigma_n-it|^{\alpha\sigma_n+\frac 12}
  \sigma_n^{|\gamma|}\left(1+|t|\right)^{|\gamma|}\;.
 \end{split}
\end{equation}
The relation~\eqref{12eyler} implies the validity of the following
inequality:
\begin{equation*}
 \begin{split}
  |\Gamma(-\sigma_n+it)|=
  &\frac{\pi}{\ch(\pi t)\,|\Gamma(1+\sigma_n-it)|}\le
  \\&
  (2\pi)^{\frac 12}  e^{\frac 16}
   e^{-\pi|t|}|\sigma_n-it|^{-\sigma_n-\frac 12}
  e^{\sigma_n}  e^{|t|\arctg\frac{|t|}{\sigma_n}}\;,
  \end{split}
\end{equation*}
for all $n\ge  n(\varepsilon,\gamma,\alpha)$ and $t\in\R$. Both
the relation~\eqref{26eyler} and the last inequality yield that,
for any $n\ge n(\varepsilon,\gamma,\alpha)$, $t\in\R$ and $0<a\le
1$, the inequality
\begin{equation}\label{27eyler}
 \begin{split}
   &|\zeta(-\alpha\sigma_n-\gamma+i\alpha t,a)\,\Gamma(-\sigma_n+it)|\le
  \\
  &\frac{C_1(\gamma)}{(2\pi)^{\alpha\sigma_n}}\,
  e^{\psi(t)+(1-\alpha)\sigma_n}
  (\alpha+\varepsilon)^{\alpha\sigma_n+\gamma+\frac 12}
  |\sigma_n-it|^{(\alpha-1)\sigma_n}
  \sigma_n^{|\gamma|}\left(1+|t|\right)^{|\gamma|}
 \end{split}
\end{equation}
is fulfilled, where $C_1(\gamma)=(2\pi)^{\frac 12} e^{\frac 16}C(\gamma)$ and
\begin{equation*}
\begin{split}
&\psi(t)=
-\pi|t|+|t|\arctg\frac{|t|}{\sigma_n}+\frac{\pi}{2}\alpha|t|-
\alpha|t|\arctg\frac{\alpha|t|}{\alpha\sigma_n+\gamma}=
-\frac{\pi}{2}|t|-\\
&-(1-\alpha)|t|\left(\frac{\pi}{2}-\arctg\frac{|t|}{\sigma_n}\right)+
\alpha|t|\left(\arctg\frac{|t|}{\sigma_n}-
\arctg\frac{\alpha|t|}{\alpha\sigma_n+\gamma}\right) \le
  \\& -\frac{\pi}{2}|t|+|\gamma|\;.
\end{split}
\end{equation*}
Here we take account of $0<\alpha\le 1$ and use the inequality
$$
|\arctg v-\arctg
u|=\left|\int_{u}^{v}\frac{dx}{x^2+1}\right|\le\left|\int_{u}^{v}\frac{dx}{x^2}\right|=\left|\frac
1u-\frac 1v\right|\,,\,u,v>0\,.
$$
Combining the inequality $|\sigma_n-it|^{(\alpha-1)\sigma_n}\le
\sigma_{n}^{(\alpha-1)\sigma_n}$ in~\eqref{27eyler} with the
estimate for $\psi(t)$ yields the validity of the relation:
\begin{equation*}
 \begin{split}
   &|\zeta(-\alpha\sigma_n-\gamma+i\alpha t,a)\,\Gamma(-\sigma_n+it)|\le
  \\
  &\frac{C_1(\gamma) e^{|\gamma|}}{(2\pi)^{\alpha\sigma_n}}\,
  e^{ -\frac{\pi}{2}|t|}
  e^{(1-\alpha)\sigma_n}
  (\alpha+\varepsilon)^{\alpha\sigma_n+\gamma+\frac 12}
  \sigma_{n}^{(\alpha-1)\sigma_n}
  \sigma_n^{|\gamma|}\left(1+|t|\right)^{|\gamma|}\;
 \end{split}
\end{equation*}
for every $n\ge  n(\varepsilon,\gamma,\alpha)$, $t\in\R$ and
$0<a\le 1$. Applying this inequality to~\eqref{23eyler} we obtain
the following estimate for $I_n$ in~\eqref{17eyler} in the case of
$n\ge  n(\varepsilon,\gamma,\alpha)$:
\begin{equation}\label{28eyler}
\left.
 \begin{split}
   &|I_n|\le
    \frac{x^{\sigma_n}
    e^{(1-\alpha)\sigma_n}(\alpha+\varepsilon)^{\alpha\sigma_n+\gamma+\frac
    12}
    \sigma_{n}^{(\alpha-1)\sigma_n+|\gamma|}}{(2\pi)^{\alpha\sigma_n+1}}\,
    I(\gamma)\,,
    \\
    & I(\gamma)= C_1(\gamma) e^{|\gamma|} \int_{-\infty}^{+\infty}e^{
    -\frac{\pi}{2}|t|}\left(1+|t|\right)^{|\gamma|}\,dt\,.
 \end{split}
 \right\}
\end{equation}
If $0<\alpha<1$ then~\eqref{28eyler} implies that
$\lim\limits_{n\to\infty}I_n=0$ for any $x>0$. If $\alpha=1$ then
$\lim\limits_{n\to\infty}I_n=0$ for every
$x\in(0,\frac{2\pi}{1+\varepsilon})$ and hence for every
$x\in(0,2\pi)$. Thus the second part of Theorem~\ref{th4} is
proved in the case $0<a\le 1$.

If $a> 1$ then $a=p+a_0$, where $p\in\N$, $0<a_0\le 1$, and we
find:
\begin{equation*}
|\zeta(-\alpha\sigma_n-\gamma+i\alpha t,a)|\le
|\zeta(-\alpha\sigma_n-\gamma+i\alpha
t,a_0)|+a(a-1)^{\alpha\sigma_n+\gamma}\;
\end{equation*}
(see~\eqref{09eyler}). In this case the right-hand side of the
inequality~\eqref{28eyler} contains one more summand:
\begin{equation}\label{29eyler}
\begin{split}
  &\frac{x^{\sigma_n}}{\sqrt{2\pi}}
  \int_{-\infty}^{+\infty}
   e^{-\pi|t|}|\sigma_n-it|^{-\sigma_n-\frac 12}
  e^{\sigma_n+\frac 16}
  e^{|t|\arctg\frac{|t|}{\sigma_n}}a(a-1)^{\alpha\sigma_n+\gamma}
  \,dt\le
  \\
  &\frac{x^{\sigma_n}}{\sqrt{2\pi}}\,
  \sigma_{n}^{-\sigma_n-\frac 12}
  e^{\sigma_n+\frac 16}
  a(a-1)^{\alpha\sigma_n+\gamma}
  \int_{-\infty}^{+\infty}
  e^{-\frac{\pi}{2}|t|}\,dt\;,\;n\ge
  n(\varepsilon,\gamma,\alpha)\;.
  \end{split}
  \end{equation}
Here we use the inequality $|\sigma_n-it|^{-\sigma_n-\frac 12}\le
\sigma_{n}^{-\sigma_n-\frac 12}$. The right-hand side of the
inequality~\eqref{29eyler} tends to $0$ as $n\to\infty$ for any
$x>0$. Theorem~\ref{th4} is complete.

\subsection*{Proof of Theorem~\ref{th5}}
The proof follows from both Theorem~\ref{th4} and the obvious
relation
$$
\widetilde{f}(x,a,\gamma,\alpha)={f}(x,a,\gamma,\alpha)-
2^{\gamma+1}{f}\left(2^{\alpha}x,\frac{a+1}{2},\gamma,\alpha\right)\,,\,
x>0\,.
$$

\subsection*{Proof of Theorem~\ref{thg1}}
Take $\sigma=\beta\in\left(\max\left\{0,
\frac{\gamma+1}{\alpha}\right\},\mu\right)$ in~\eqref{melin2}, and
replace $x$ by $(k+a)^{\alpha}x$, $a>0$, $k\in\Z_+$, $\alpha>0$,
$x>0$. Then summarize the obtained relations
\begin{equation*}
\frac{(k+a)^\gamma}{(x(k+a)^\alpha+1)^\mu}=\frac{1}{2\pi i}
\int_{\beta-i\infty}^{\beta+i\infty}
\frac{\Gamma(\mu-s)\Gamma(s)}{\Gamma(\mu)}x^{-s}(k+a)^{-\alpha
s+\gamma}\,ds
\end{equation*}
over all $k\in\Z_+$. In the left-hand side, we obtain
$g(x,a,\gamma,\alpha,\mu)$. In the right-hand side, we interchange
the sum and integral signs  (this is well defined in view
of~\eqref{11eyler_a},~\eqref{12eyler} and $\alpha\beta-\gamma>1$).
It follows that, for any $a>0$, $\gamma\in \R$,  $\alpha>0$,
$\mu>\max\left\{0,\frac{\gamma+1}{\alpha}\right\}$, $x>0$ and
$\max\left\{0,\frac{\gamma+1}{\alpha}\right\}<\beta<\mu$, the
following relation holds:
\begin{equation*}
g(x,a,\gamma,\alpha,\mu)= \frac{1}{2\pi i}
\int_{\beta-i\infty}^{\beta+i\infty}G(s)\,ds\;,
\end{equation*}
where $G(s)=\frac{\Gamma(\mu-s)\Gamma(s)}{\Gamma(\mu)}
x^{-s}\zeta(\alpha
s-\gamma,a)=\frac{\Gamma(\mu-s)}{\Gamma(\mu)}F(s)$ and the
function $F$ from~\eqref{15eyler}.  Singular points of the
functions $G$ and $F$ coincide on the half-plane $\Re s<\mu$.
Consider $\sigma_n=n+\frac 12$, $n\in\Z_+$. As in the proof of
Theorem~\ref{th4}, taking  account of~\eqref{13eyler},
\eqref{14eyler} and the equality~\eqref{09eyler} (for $a>1$), in
the case of $\sigma_n\ne-\frac{\gamma+1}{\alpha}$ we obtain the
following relation:
\begin{equation}\label{b7a}
g(x,a,\gamma,\alpha,\mu)=\Sigma_n+I_n\,,\, I_n=\frac{1}{2\pi i}
\int_{-\sigma_n-i\infty}^{-\sigma_n+i\infty}G(s)\,ds\,,
\end{equation}
where $\Sigma_n$ stands for the sum of residues of the function
$G$ at poles lying on the interval $(-\sigma_n,\beta)$. If
$n>-\frac 12-\frac{\gamma+1}{\alpha}$, $n\in\Z_+$, then the
interval $(-\sigma_n,\beta)$ contains only the poles $s=-k$,
$k=0,\ldots,n$, and $s=\frac{\gamma+1}{\alpha}$. If
$-\frac{\gamma+1}{\alpha}\not\in\Z_+$ then these poles are simple,
and we have:
\begin{equation*}
\begin{split}
&\underset{s=\frac{\gamma+1}{\alpha}}\res
G(s)=\frac{\Gamma\left(\mu-\frac{\gamma+1}{\alpha}
\right)}{\Gamma(\mu)}\,\underset{s=\frac{\gamma+1}{\alpha}}\res
F(s) =\frac{\Gamma\left(\mu-\frac{\gamma+1}{\alpha}\right)\Gamma
\left(\frac{\gamma+1}{\alpha}\right)}{\Gamma(\mu)\,\alpha}\,
x^{-\frac{\gamma+1}{\alpha}}\,,
   \\
&\underset{s=-k}\res G(s)=\frac{\Gamma(\mu+k)}{\Gamma(\mu)}\,
\frac{(-1)^k}{k!}\,\zeta(-\alpha k-\gamma,a)\,x^k\;,\;k\in\Z_+\;.
\end{split}
\end{equation*}
For $-\frac{\gamma+1}{\alpha}=r\in\Z_+$, the residues at the
points $s=-k$, $k\in\Z_+$, $k\ne r$, can be calculated as above.
To calculate the residue of the function $G$ at the point $s=-r$
it should be considered~\eqref{21eyler_a}, \eqref{22eyler} as well
as the following expansion of the function
$\frac{\Gamma(\mu-s)}{\Gamma(\mu)}$ in a Taylor series in a
neighborhood of the point $s=-r$:
\begin{equation*}
\frac{\Gamma(\mu-s)}{\Gamma(\mu)}=A_0+A_1(s+r)+\ldots\;,
\;A_0=\frac{\Gamma(\mu+r)}{\Gamma(\mu)}\;,\;
A_1=-\frac{\Gamma'(\mu+r)}{\Gamma(\mu)}\,\;.
\end{equation*}
\begin{equation*}
\begin{split}
&\underset{s=-r}\res G(s)=B_{-1}A_{0}+B_{-2}A_{1}=
\\
&\frac{\Gamma(\mu+r)(-1)^rx^r}{\Gamma(\mu)\Gamma(r+1)}\left(-\frac{\ln
x}{\alpha}+\frac{\Gamma'(r+1)}{\alpha\,\Gamma(r+1)}-
\frac{\Gamma'(a)}{\Gamma(a)}-
\frac{\Gamma'(\mu+r)}{\alpha\,\Gamma(\mu+r)}\right)\;.
\end{split}
\end{equation*}
Let us find estimate for the integral $I_n$ in~\eqref{b7a}. If
$n\in\Z_+$, $\sigma_n\ne-\frac{\gamma+1}{\alpha}$ and $x>0$, we
have:
\begin{equation}\label{b9}
|I_n|\le\frac{x^{\sigma_n}}{2\pi\Gamma(\mu)}
\int_{-\infty}^{+\infty}
  |\Gamma(\mu+\sigma_n-it)|\,
  |\Gamma(-\sigma_n+it)|\,|\zeta(-\alpha\sigma_n-\gamma+i\alpha t,a)|\,dt\;.
\end{equation}
The convergence of the integral in~\eqref{b9} is proved exactly in
the same way as that for the integral in~\eqref{23eyler}.
Theorem~\ref{thg1} is complete.

\subsection*{Proof of Theorem~\ref{thg2}}
If $\mu>\max\left\{0,\frac{\gamma+1}{\alpha}\right\}$ then the
proof is implied both by Theorem~\ref{thg1} and the relation
$$
\widetilde{g}(x,a,\gamma,\alpha,\mu)={g}(x,a,\gamma,\alpha,\mu)-
2^{\gamma+1}{g}\left(2^{\alpha}x,\frac{a+1}{2},\gamma,\alpha,\mu\right)\,,\,
 x>0\,.
$$
If $\mu>\max\left\{0,\frac{\gamma}{\alpha}\right\}$, the proof is
exactly the same as that of Theorem~\ref{thg1}. It should be noted
that in this case, for any $a>0$, $\gamma\in \R$, $\alpha>0$,
$\mu>\max\left\{0,\frac{\gamma}{\alpha}\right\}$, $x>0$ and
$\max\left\{0,\frac{\gamma}{\alpha}\right\}<\beta<\mu$ the
relation
\begin{equation*}
\widetilde{g}(x,a,\gamma,\alpha,\mu)= \frac{1}{2\pi i}
\int_{\beta-i\infty}^{\beta+i\infty}\widetilde{G}(s)\,ds\;
\end{equation*}
holds, where $\widetilde{G}(s)=
\frac{\Gamma(\mu-s)\Gamma(s)}{\Gamma(\mu)}x^{-s}\widetilde{\zeta}(\alpha
s-\gamma,a)$.

\section{Proofs of Theorems~\ref{th1}, \ref{th2} and~\ref{th3}}
\label{subsec5}
\begin{lemma}\label{le}
Suppose that $a>0$, $\gamma\in \R$ and $\alpha>0$. Then there are
no constants $p,\beta,c\in\R$ such that  one of the identities
$x^\beta e^{p\,x}{f}(x,a,\gamma,\alpha)\equiv c$ or $x^\beta
e^{p\,x}\widetilde{f}(x,a,\gamma,\alpha)\equiv c$ holds for $x>0$.
\end{lemma}

\begin{proof}[Proof]
Assume that $x^\beta e^{p\,x}{f}(x,a,\gamma,\alpha)\equiv c$,
$x>0$. Then $c>0$, and it follows from the asymptotics
$f(x,a,\gamma,\alpha)\sim a^{\gamma}e^{-a^\alpha x}$,
$x\to+\infty$ that $p=a^\alpha$ (if $p>a^\alpha$ or $p<a^\alpha$
then $c=+\infty$ or $c=0$ respectively, which is impossible).
Hence $\beta=0$ (if $\beta>0$ or $\beta<0$ then   $c=+\infty$ or
$c=0$, which is impossible) and $c=a^\gamma$. Therefore
${f}(x,a,\gamma,\alpha)\equiv a^\gamma e^{-a^\alpha\,x}$, $x>0$
but ${f}(x,a,\gamma,\alpha)> a^\gamma e^{-a^\alpha\,x}$ for every
$x>0$.

Assume that $x^\beta
e^{p\,x}\widetilde{f}(x,a,\gamma,\alpha)\equiv c$, $x>0$. It
follows from the asymptotics $\widetilde{f}(x,a,\gamma,\alpha)\sim
a^{\gamma}e^{-a^\alpha x}$, $x\to+\infty$ that $c>0$. As above we
obtain similarly that  $\widetilde{f}(x,a,\gamma,\alpha)\equiv
a^\gamma e^{-a^\alpha\,x}$, $x>0$. Then
$\widetilde{f}(x,a+1,\gamma,\alpha)\equiv a^\gamma
e^{-a^\alpha\,x}-\widetilde{f}(x,a,\gamma,\alpha)\equiv 0$, $x>0$
but $\widetilde{f}(x,a+1,\gamma,\alpha)>0$ for large $x>0$.
Lemma~\ref{le} is proved.
\end{proof}

\subsection*{Proof of Theorem~\ref{th1}}
In the case of $a>0$, $\gamma+1>0$,  $\alpha>0$,
$\mu>\frac{\gamma+1}{\alpha}$, $p\ge0$ and $c\in\R$, define the
following function of the variable $x>0$:
$$
\varphi(x,\agam,c,p):=\frac{c}{(p+x)^{\mu-\frac{\gamma+1}{\alpha}}}
\cdot\frac{\Gamma\left(\mu-\frac{\gamma+1}{\alpha}\right)}{\Gamma(\mu)}-
S(x,\agam)\,.
$$
It is easy to check that for any $x>0$ and $k\in\Z_+$ there hold the relations:
\begin{equation}\label{c14}
\left.
\begin{split}
&(-1)^k\frac{d^{\,k}}{dx^k}\left\{\varphi(x,\agam,c,p)\right\}=
\\&
\frac{\Gamma(\mu+k)}{\Gamma(\mu)}\varphi(x,\agam+k,c,p)\,,
\\&
  \varphi(x,\agam,c,p)=
  \\&
  \frac{1}{\Gamma(\mu)}\int_{0}^{+\infty}e^{-xt}t^{\mu-\frac{\gamma+1}{\alpha}-1}
  \left(c\,e^{-pt}-t^{\frac{\gamma+1}{\alpha}}f(t,a,\gamma,\alpha)\right)\,dt\;.
\end{split}
\right\}
\end{equation}
The integral representation in~\eqref{c14} follows from both the
inequalities $\mu>\frac{\gamma+1}{\alpha}$, $p\ge0$ and the
asymptotics $f(t,a,\gamma,\alpha)\sim a^{\gamma}e^{-a^\alpha t}$,
$t\to+\infty$ and $f(t,a,\gamma,\alpha)\sim
\frac{1}{\alpha}\Gamma\left(\frac{\gamma+
1}{\alpha}\right)t^{-\frac{\gamma+1}{\alpha}}$, $t\to+0$ (see
Theorem~\ref{th4}). These asymptotics imply also that
$A_p(a,\gamma,\alpha)<+\infty$ if and only if $p<a^\alpha$, and
$B_q(a,\gamma,\alpha)>0$ if and only if $q\ge a^\alpha$. We will
need the following theorem in what follows
(see~\cite{Bernstein,Hausdorff,Widder}).

\begin{theorem}[Bernstein--Hausdorff--Widder]\label{old}
The following two conditions are equivalent:
\begin{enumerate}
\item A function $f\in  C^{\infty}{(0,+\infty)}$, and the
inequality $(-1)^kf^{(k)}(x)\ge 0$ holds for all $k\in\Z_+$,
$x>0$.

\item  We have $f(x)=\int_0^{+\infty}e^{-xt}\ d\mu(t)$, $ x>0$,
where $\mu$ is a nonnegative Borel measure on $[0,+\infty)$ such
that the integral converges for every $x>0$. In this case the
measure $\mu$ is finite on $[0,+\infty)$ if and only if
$f(+0)<+\infty$.
  \end{enumerate}
  \end{theorem}

Both the Berstein--Hausdorff--Widder theorem and
relations~\eqref{c14} yield that the validity of
inequalities~\eqref{n2} for every $\mu>\mu_0$ and $x>0$ is
equivalent to that for the inequalities
$Ae^{-pt}-t^{\frac{\gamma+1}{\alpha}}f(t,a,\gamma,\alpha)\ge 0$
and $Be^{-qt}-t^{\frac{\gamma+1}{\alpha}}f(t,a,\gamma,\alpha)\le
0$ for any $t>0$. The last ones are equivalent to the inequalities
$A\ge A_p(a,\gamma,\alpha)$ and $B\le B_q(a,\gamma,\alpha)$
respectively.

The relation~\eqref{ass1} implies that $B_q(a,\gamma,\alpha)\le
\frac{\Gamma\left(\frac{\gamma+1}{\alpha}\right)}{\alpha} \le
A_p(a,\gamma,\alpha)$ for $0\le p<a^2\le q$. If $a\ge 1$ (see the
inequality~\eqref{ner2} and the text below it), we have $\frac
12\le B_q(a,1,2)$, $A_p(a,1,2)\le\frac 12$ for $0\le p\le
m(\infty,a)$, $q\ge a^2$, and hence $B_q(a,1,2)=A_p(a,1,2)=\frac
12$ for any $a\ge 1$ and for the same $p$ and $q$; in particular,
this is satisfied for $p\in[0,a^2-a]$ (since $a^2-a<m(\infty,a)$).

Let $0\le p<a^\alpha$,  $A\ge A_p(a,\gamma,\alpha)$, and let $q\ge
a^\alpha$, $B\le B_q(a,\gamma,\alpha)$. If either the right-hand
side or the left-hand side inequality in~\eqref{n2} turns into an
equality for some $x\ge 0$ (or for $x>0$ if $p=0$) then it follows
from the integral representation~\eqref{c14} that either $A\equiv
e^{pt}t^{\frac{\gamma+1}{\alpha}}f(t,a,\gamma,\alpha)$ or $B\equiv
e^{qt}t^{\frac{\gamma+1}{\alpha}}f(t,a,\gamma,\alpha)$ for $t>0$,
which contradicts Lemma~\ref{le}. Theorem~\ref{th1} is complete.

\subsection*{Proof of Theorem~\ref{th2}}
For $a>0$, $\gamma+1<0$,  $\alpha>0$,  $\mu>0$, $p\ge0$ and
$c\in\R$, define the function
$$
\psi(x,\agam,c,p):=\frac{c}{(p+x)^{\mu}}-  S(x,\agam)\,,\, x>0\,.
$$
It is easy to check that for any $x>0$ and $k\in\Z_+$ there hold the inequalities:
\begin{equation}\label{c15}
\left.
\begin{split}
&(-1)^k\frac{d^{\,k}}{dx^k}\left\{\psi(x,\agam,c,p)\right\}=
\frac{\Gamma(\mu+k)}{\Gamma(\mu)}\psi(x,\agam+k,c,p)\,,
\\&
  \psi(x,\agam,c,p)=\frac{1}{\Gamma(\mu)}\int_{0}^{+\infty}e^{-xt}t^{\mu-1}
  \left(c\,e^{-pt}-f(t,a,\gamma,\alpha)\right)\,dt\;.
\end{split}
\right\}
\end{equation}
The integral representation in~\eqref{c15} is implied by both the
inequalities $\mu>0$, $p\ge0$, the asymptotics
$f(t,a,\gamma,\alpha)\sim a^{\gamma}e^{-a^\alpha t}$,
$t\to+\infty$ as well as by the equality
$f(+0,a,\gamma,\alpha)=\zeta(-\gamma,a)>0$ (see
Theorem~\ref{th4}). These relations yield also that
$D_p(a,\gamma,\alpha)<+\infty$ if and only if $p\le a^\alpha$, and
$E_q(a,\gamma,\alpha)>0$ if and only if $q\ge a^\alpha$. If $p\le
a^\alpha$ then the function $e^{pt}f(t,a,\gamma,\alpha)$ strictly
decreases with respect to $t>0$. Therefore
$D_p(a,\gamma,\alpha)=\zeta(-\gamma,a)$ for all $p\le a^\alpha$
and $E_{a^\alpha}(a,\gamma,\alpha)=a^\gamma$.

Both the Bernstein--Hausdorff--Widder theorem and the
relations~\eqref{c15} imply that the validity of
inequalities~\eqref{n3} for every $\mu>\mu_0$ and $x>0$ is
equivalent to that for the inequalities
$De^{-pt}-f(t,a,\gamma,\alpha)\ge 0$ and
$Ee^{-qt}-f(t,a,\gamma,\alpha)\le 0$ for any $t>0$. The last ones
are equivalent to the inequalities $D\ge D_p(a,\gamma,\alpha)$ and
$E\le E_q(a,\gamma,\alpha)$ respectively.

Let $0\le p\le a^\alpha$, $D\ge D_p(a,\gamma,\alpha)$, and let
$q\ge a^\alpha$, $E\le E_q(a,\gamma,\alpha)$. If either the
right-hand side or the left-hand side inequality in~\eqref{n3}
turns into an equality for some $x\ge 0$ (or for $x> 0$ if $p=0$),
then the integral representation~\eqref{c15} yields that either
$D\equiv e^{pt}f(t,a,\gamma,\alpha)$ or $E\equiv
e^{qt}f(t,a,\gamma,\alpha)$ for $t>0$, which is impossible (see
Lemma~\ref{le}). Theorem~\ref{th2} is complete.

\subsection*{Proof of Theorem~\ref{th3}}
For $a>0$, $\gamma\in \R$, $\alpha>0$,
$\mu>\max\{\frac{\gamma}{\alpha};0\}$, $p\ge0$ and $c\in\R$ define
the function
$$
\widetilde{\psi}(x,\agam,c,p):=\frac{c}{(p+x)^{\mu}}-
  \widetilde{S}(x,\agam)\,,\, x>0\,.
$$
It is easy to check that for any $x>0$ and $k\in\Z_+$ there hold the relations:
\begin{equation}\label{c16}
\left.
\begin{split}
&(-1)^k\frac{d^{\,k}}{dx^k}\left\{\widetilde{\psi}(x,\agam,c,p)\right\}=
\frac{\Gamma(\mu+k)}{\Gamma(\mu)}\widetilde{\psi}(x,\agam+k,c,p)\,,
\\&
  \widetilde{\psi}(x,\agam,c,p)=\frac{1}{\Gamma(\mu)}\int_{0}^{+\infty}e^{-xt}t^{\mu-1}
  \left(c\,e^{-pt}-\widetilde{f}(t,a,\gamma,\alpha)\right)\,dt\;.
\end{split}
\right\}
\end{equation}
The integral representation in~\eqref{c16} follows from both the
inequalities  $\mu>\max\{\frac{\gamma}{\alpha};0\}$, $p\ge0$, the
asymptotics  $\widetilde{f}(t,a,\gamma,\alpha)\sim
a^{\gamma}e^{-a^\alpha t}$, $t\to+\infty$ as well as from the
equality
$\widetilde{f}(+0,a,\gamma,\alpha)=\widetilde{\zeta}(-\gamma,a)$
(see Theorem~\ref{th5}). These relations yield also that
$0<C_p(a,\gamma,\alpha)<+\infty$ for $p\le a^\alpha$ and that
$C_p(a,\gamma,\alpha)=+\infty$ for $p> a^\alpha$, and that
$F_q(a,\gamma,\alpha)>-\infty$ as well.

Both the Bernstein--Hausdorff--Widder theorem and the
equalities~\eqref{c16} imply that the validity of
inequalities~\eqref{n1} for any $\mu>\mu_0$ and $x>0$ is
equivalent to that for the inequalities
$Ce^{-pt}-\widetilde{f}(t,a,\gamma,\alpha)\ge 0$ and
$Fe^{-qt}-\widetilde{f}(t,a,\gamma,\alpha)\le 0$ for any $t>0$.
The last ones are equivalent to the inequalities $C\ge
C_p(a,\gamma,\alpha)$ and $F\le F_q(a,\gamma,\alpha)$
respectively.

Let $0\le p\le a^\alpha$, $C\ge C_p(a,\gamma,\alpha)$, and let
$q\ge 0$, $F\le  F_q(a,\gamma,\alpha)$. If either the right-hand
side or the left-hand side inequality in~\eqref{n1} becomes the
equality for some $x\ge 0$ (or for $x>0$ if $p=0$ or $q=0$) then
the  representation~\eqref{c16} implies that either $C\equiv
e^{pt}\widetilde{f}(t,a,\gamma,\alpha)$ or $F\equiv
e^{qt}\widetilde{f}(t,a,\gamma,\alpha)$ for $t>0$, which is
impossible due to Lemma~\ref{le}. Theorem~\ref{th3} is complete.


\noindent{\bf Acknowledgements.} This work has been supported by the
Foundation for Fundamental Research of Ukraine, Grant F25.1/055.

\hfill{}Translated from Russian by D. Lymanskyi

\bigskip\bigskip

CONTACT INFORMATION

\medskip
  V.P. Zastavnyi\\
  Donetsk National University \\
  Universitetskaya str., 24 \\
  Donetsk 83001 Ukraine \\
  zastavn@rambler.ru


\begin{thebibliography}{30}


\bibitem{Timan}
A. F. Timan, \textit{A precise estimate of the remainder in the
approximation of differentiable functions by Poisson integrals},
Dokl. Akad. Nauk SSSR \textbf{74} (1), 17--20 (1950).


\bibitem{Maley}
L. V. Malei, \textit{Precise estimate for approximation of
quasi-smooth functions by Poisson integrals}, Vesci Akad. Navuk
BSSR Ser. Fiz.-Tehn. Navuk, No.3, 25--32 (1961).


 \bibitem{Stark}
\'{E}. L. Shtark, \textit{The complete asymptotic expansion for
the supremum of the deviation of functions in ${\rm Lip} 1$ from
their Abel--Poisson singular integrals}, Mat. Zametki \textbf{13}
(1), 21--28 (1973).


\bibitem{Baskakov}
V. A. Baskakov, \textit{Certain properties of operators of
Abel--Poisson type}, Mat. Zametki \textbf{17} (2), 169--180
(1975).


\bibitem{Harkevich2002_1}
K. M. Zhigallo and Yu. I. Kharkevich, \textit{The complete
asymptotics of the deviation from the class of differentiable
functions of the set of their Poisson harmonic integrals}, Ukrain.
Mat. Zh.  \textbf{54} (1), 43--52 (2002) [Ukrainian Math.
\textbf{54} (1), 51--63 (2002)].


\bibitem{Zast2009_mz}
V. P. Zastavnyj, \textit{On series arising from the approximation
of periodic differentiable functions by Poisson integrals}, Mat.
Zametki \textbf{86} (4), 497--511 (2009).


\bibitem{Gelfond}
A. O. Gel'fond, \textit{Residues and Their Applications}, Nauka,
Moscow, 1966.


\bibitem{Zast2008}
V. P. Zastavnyj, \textit{The generalized Euler--Maclaurin formula
and its application}, Trudy Inst. Appl. Math. Mech. \textbf{17},
51--60 (2008).


\bibitem{UW}
E. T. Whittaker, G. N. Watson, \textit{A Course of Modern
Analysis}, Cambridge, Cambridge University Press, 1927.


\bibitem{Bateman}
H. Bateman and A. Erd\'{e}lyi, \textit{Higher Transcendental
Functions}, Vol. 1: \textit{The Hypergeometric Function, Legendre
Functions} (McGraw--Hill, New York--Toronto--London, 1953; Nauka,
Moscow, 1965 and 1973 (2nd ed.))


\bibitem{Riekst}
E. A. Riekstyn'sh, \textit{Estimates for Remainders in Asymptotic
Expansions}, Zinatne, Riga, 1986.


\bibitem{Titch}
E. Titchmarsh, \textit{Introduction to the Theory of Fourier
Integrals}, 3rd Ed., Chelsea, New York, 1986.


\bibitem{Mathieu}
\'{E}. L. Mathieu, \textit{Trait\'{e} de Physique
Math\'{e}matique. VI-VII: Th\'{e}ory de l’\'{E}lasticit\'{e} des
Corps Solides (Part~2)}, Paris, Gauthier-Villars, 1890.


\bibitem{Berg}
L. Berg, \textit{\"Uber eine Absch\"atzung von Mathieu},  Math.
Nachr., \textbf{7}, 257--259 (1952).


\bibitem{Corput}
J. G. {van der} Corput,  L. O. Heflinger, \textit{On the
inequality of Mathieu}, Indagationes Mathematicae, \textbf{18},
15--20 (1956).


\bibitem{Makai}
E. Makai, \textit{On the inequality of Mathieu}, Publ. Math.
Debrecen, \textbf{5}, 204--205 (1957).


\bibitem{Elbert}
A. Elbert, \textit{Asymptotic expansion and continued fraction for
Mathieu’s series}, Period. Math. Hungar., \textbf{13}, 1--8
(1982).


\bibitem{Alzer}
H. Alzer, J. L. Brenner, O. G. Ruehr, \textit{On Mathieu's
inequality}, J. Math. Anal. Appl., \textbf{218}, 607--610 (1998).


\bibitem{Zast2009_1}
V. P. Zastavnyi, \textit{Mathieu's series: inequalities,
asymptotics and positive definiteness},
http://arxiv.org/abs/0901.1104v1  (2009).


\bibitem{Diananda}
P. H. Diananda, \textit{Some Inequalities Related to an Inequality
of Mathieu}, Math. Ann., \textbf{250}, 95--98 (1980).


\bibitem{Hoorfar and Qi}
A. Hoorfar, F. Qi, \textit{Some new bounds for Mathieu's series},
Abstract and Applied Analysis, \textbf{2007}, article ID 94854, 10
pages (2007).


\bibitem{Tomovski_Hilfer}
\v{Z}. Tomovski, R. Hilfer, \textit{Some bounds for alternating
Mathieu type series}, Journal of Mathematical Inequalities,
\textbf{2} (1), 17--26 (2008).


\bibitem{Zast2009}
V. P. Zastavnyi, \textit{On a paper of \v{Z}. Tomovski and R.
Hilfer}, http://arxiv.org/abs/0901.4766v1  (2009).


\bibitem{Bernstein}
S. N. Bernstein, \textit{Sur les fonctions absolument monotones},
Acta Math., \textbf{52} (1), 1--66 (1929).


 \bibitem{Hausdorff}
F. Hausdorff, \textit{Summationsmethoden und Momentfolgen. II},
Math. Zeitschrift, \textbf{9}, 280--299 (1921).


\bibitem{Widder}
D. V. Widder, \textit{Necessary and sufficient conditions for the
representation of a function as a Laplace integral},
 Trans. Amer. Math. Soc., \textbf{33} (4), 851--892 (1931).


\end{thebibliography}
\end{document}